\documentclass[12pt]{amsart}
\usepackage{amsmath,amssymb,amsbsy,amsfonts,
latexsym,amstext,amsxtra,xcolor,fullpage}
\usepackage[utf8]{inputenc}

\begin{document}

\newtheorem{theorem}{Theorem}
\newtheorem{lemma}[theorem]{Lemma}
\newtheorem{claim}[theorem]{Claim}
\newtheorem{cor}[theorem]{Corollary}
\newtheorem{prop}[theorem]{Proposition}
\newtheorem{example}[theorem]{Example}
\newtheorem{definition}{Definition}
\newtheorem{quest}[theorem]{Question}
\newtheorem{problem}[theorem]{Problem}


\def\eps{\varepsilon}
\def\al{\alpha}
\def\be{\beta}
\def\ga{\gamma}
\def\ro{\varrho}
\def\N{\mathbb{N}}
\def \Z{{\mathbb Z}}
\def \Q{{\mathbb Q}}
\def \R{{\mathbb R}}
\def \C{{\mathbb C}}

\title[Extremal problems for polynomials with real roots]{Extremal problems for polynomials with real roots}

\author{Art\= uras Dubickas}
\address{Institute of Mathematics, Faculty of Mathematics and Informatics, Vilnius University, Naugarduko 24,
LT-03225 Vilnius, Lithuania}
\email{arturas.dubickas@mif.vu.lt}

\author{Igor Pritsker}
\address{Department of Mathematics, Oklahoma State University, Stillwater, OK 74078, U.S.A.}
\email{igor@math.okstate.edu}

\begin{abstract}
We consider polynomials of degree $d$ with only real roots and a fixed value of discriminant, and study the problem of minimizing
the absolute value of polynomials at a fixed point off the real line. There are two explicit families of polynomials that turn out to be extremal
in terms of this problem. The first family has a particularly simple expression as a linear combination of $d$-th powers of two linear
functions. Moreover, if the value of the discriminant is not too small, then the roots of the extremal polynomial and the smallest absolute
value in question can be found explicitly. The second family is related to generalized Jacobi (or Gegenbauer) polynomials, which helps us
to find the associated discriminants.
We also investigate the dual problem of maximizing the value of discriminant, while keeping the absolute value of polynomials
at a point away from the real line fixed. Our results are then applied to problems on the largest disks contained in lemniscates, and
to the minimum energy problems for discrete charges on the real line.
\end{abstract}


\maketitle

\section{Extremal problems and their solutions} \label{ExPr}

We study polynomials of degree $d$ of the form
$$
    f(x) = \sum_{k=0}^{d} c_k x^k = c_d \prod_{k=1}^{d} (x-x_k)    \in \R[x],
$$
with $d$ {\it real} roots
$x_1,\dots,x_d$ and leading coefficient $c_d \in \R$, $c_d \ne 0$.
The discriminant of $f$ is defined by $$\Delta=\Delta_f:=c_d^{2d-2} \prod_{1\le j<k\le d} (x_j-x_k)^2.$$ It is positive if all the roots $x_k$,
$k=1,\dots,d$,
are distinct.

Earlier, various extremal problems involving discriminants of polynomials with real roots were considered by Stieltjes \cite{Sti1}-\cite{Sti3}, Schur \cite{Sch}, Siegel \cite{Sie}, and others, because of many applications of such results in analysis and number theory.
In this paper, we are interested in minimizing the absolute value of $f$ at a given point off the real line, among all polynomials with a given value of the leading coefficient $c_d=A$ and a fixed value of discriminant. Since discriminant is invariant with respect to the translation of all the roots by a real number and $|f(ai)|=|f(-ai)|$ for $a \in \R$, without restriction of generality, we may assume that we minimize the value of $|f(ai)|$ for a given $a>0$.
Note that there is no loss of generality if we assume that the polynomial $f$ is monic, since
by replacing $f$ with leading coefficient $A \ne 0$ by the monic polynomial $f/A$ its discriminant $\Delta=\Delta_f$ will be replaced
by $\Delta/A^{2d-2}$, whereas the minimum of $|f(ai)|$, say $m$, will become $m/|A|$.

Throughout, let $K(d,D)$ be the set of monic polynomials of degree $d$ with $d$ real roots and discriminant $D$.
In all what follows, we will investigate the following natural problem.

\begin{problem} \label{prob1}
Let $a>0$, $D>0$ and $d \geq 2$. Find all $f \in K(d,D)$
that realize the minimum of $|f(ai)|$.
\end{problem}

We also state the dual problem:

\begin{problem} \label{prob2}
Let $a>0$, $m>a^d$ and $d \geq 2$.  Find all monic polynomials $f$ of degree $d$ with $d$ real roots
and fixed value $|f(ai)|=m$ that have the largest possible value of discriminant.
\end{problem}
The value $|f(ai)|$ for any monic polynomial $f(x) = \prod_{k=1}^{d} (x-x_k)$ of degree $d$ with real roots is easily
estimated as follows:
\[
m=|f(ai)| = \prod_{k=1}^{d} |ai-x_k| \ge a^d.
\]
Equality holds above if and only if $x_k=0$ for all $k=1,\dots,d$. Hence, it is natural to use the restriction $m>a^d$ in the statement of
Problem \ref{prob2}.

Our first theorem solves Problem~\ref{prob1} if $a$ is not too large in terms of $d$ and $D$.

\begin{theorem} \label{BoundD}
Let $a>0$, $D>0$ and $d \geq 2$. Then, for each $f \in K(d,D)$,
we have
\begin{equation}\label{aa101}
|f(ai)| \geq (2a)^{d/2} d^{-d/(2d-2)} D^{1/(2d-2)}.
\end{equation}
If, in addition,
\begin{equation}\label{aa100}
a \leq 2^{-1+2/d} d^{-1/(d-1)} D^{1/d(d-1)},
\end{equation}
then equality in \eqref{aa101} is attained if and only if $f(x)=F(x)$ or $f(x)=(-1)^d F(-x)$, where
\begin{equation}\label{aa102}
  F(x) =F_{a,B}(x):= \frac{1}{2ad} \Big( (ad-Bi)(x+ai)^d + (ad+Bi)(x-ai)^d \Big),
\end{equation}
with
\begin{equation}\label{bbb23}
B=B(a,d,D) :=
(-1)^d ad \sqrt{a^{-d} 2^{2-d} d^{-d/(d-1)} D^{1/(d-1)}-1}.
\end{equation}
Here, the roots of $F_{a,B}$ can be expressed in the explicit form
\begin{equation}\label{aa103}
\{x_1,\dots,x_d\}=\{a \tan(\gamma + k\pi/d),\quad k=0,1,\ldots,d-1\},
\end{equation}
where $\gamma=\gamma(a,d,D)\in[0,\pi/(2d)]$ is given by
\begin{equation}\label{aa104}
\gamma =  \left\{
                 \begin{array}{ll}
                   \displaystyle\frac{\arccos p(a,d,D)}{d}, & \hbox{if $d$ is odd;} \\
                   \displaystyle\frac{\arcsin p(a,d,D)}{d}, & \hbox{if $d$ is even,}
                 \end{array}
               \right.
\end{equation}
with
\begin{equation}\label{aa1044}
p(a,d,D):=a^{d/2}2^{d/2-1} d^{d/(2d-2)} D^{-1/(2d-2)} \le 1.
\end{equation}
\end{theorem}

We remark that inequality in \eqref{aa1044} holds by \eqref{aa100}, and the expression under the square root that defines $B$ in \eqref{bbb23} is nonnegative also by \eqref{aa100}.
Note that, by \eqref{aa102}, the coefficient for $x^{d-1}$ in $F_{a,B}(x)$ is $B$. Combining this with \eqref{aa103},
we obtain the following relation between $B$ and $\gamma$:
\begin{equation}\label{bbb}
B = -a\sum_{k=0}^{d-1} \tan(\gamma +k\pi/d) = ad \cot(d\pi/2+d\gamma).
\end{equation}
Here, the last equality holds by identity (432) of \cite[pp. 80-81]{Jol}.
Also, for $B=0$, by \eqref{aa102}, we have $F_{a,0}(x)=(-1)^d F_{a,0}(-x)$, so equality in \eqref{aa101} is attained by the unique polynomial
$f=F_{a,0}$, when the upper bound for $a$ in \eqref{aa100} is attained.

We also state the companion of Theorem~\ref{BoundD} for Problem~\ref{prob2}.

\begin{theorem} \label{Vfixed}
Suppose that $a>0$, $d \geq 2$ and $m>a^d$. Then every monic polynomial $f$ of degree $d$ with $d$ real roots
and fixed value $|f(ai)|=m$ satisfies
\begin{equation}\label{DiscBound}
\Delta_f \leq \frac{m^{2d-2} d^d}{(2a)^{d(d-1)}}.
\end{equation}
If, in addition,
\begin{equation}\label{Vcond}
m \geq 2^{d-1} a^d,
\end{equation}
then equality in \eqref{DiscBound} is attained if and only if $f(x)=F(x)$ or $f(x)=(-1)^d F(-x)$, where $F=F_{a,B}$ is defined
in \eqref{aa102} and its roots are given by \eqref{aa103}.
\end{theorem}

Theorems~\ref{BoundD} and \ref{Vfixed} are equivalent. We prove Theorem~\ref{BoundD} and then derive from it
Theorem~\ref{Vfixed}. However, using our argument one can do it the other way around.

\medskip
Consider now the alternative case $a^d<m \leq 2^{d-1}a^d$, which is not covered by Theorem~\ref{Vfixed}. In that case the answer to Problem~\ref{prob2}
is given in terms of the discriminant $\Delta_G$ of the polynomial
\begin{align}\label{genfff}
  G(x) =G_{a,\lambda}(x) := x^d + \sum_{k=1}^{\lfloor d/2 \rfloor} \left((-1)^ka^{2k} {d\choose 2k} \frac{(2k-1)!!}{\prod_{j=1}^{k} (\lambda-2d+2j+1)} \right) x^{d-2k}
\end{align}
for some $\lambda=\lambda_0$. This family of polynomials is directly related to Jacobi (or Gegenbauer) polynomials with parameters outside the classical range.
They appeared in the literature several times under different names like pseudo-Jacobi, twisted Jacobi or Romanovski-Routh polynomials; see, for instance, \cite{koek}
and a recent survey \cite{Wong}. More details about this connection are given in Section~\ref{Aux22}.

\begin{theorem}\label{otherpart}
Suppose that $a>0$, $d \geq 2$ and
\begin{equation}\label{Vucond}
a^d< m \leq 2^{d-1}a^d.
\end{equation}
Then every monic polynomial $f$ of degree $d$ with $d$ real roots
and fixed value $|f(ai)|=m$ satisfies
\begin{equation}\label{DiscBound1}
\Delta_f \leq \Delta_G,
\end{equation}
where $G=G_{a,\lambda_0}$ is the polynomial defined in \eqref{genfff} with  a
unique
$\lambda_0=\lambda_0(a,d,m) \geq 2d-2$ satisfying
\begin{align}\label{genfff2}
 1+\sum_{k=1}^{\lfloor d/2 \rfloor} {d\choose 2k} \frac{(2k-1)!!}{\prod_{j=1}^{k} (\lambda_0-2d+2j+1)} =\frac{m}{a^d}.
\end{align}
Moreover, equality in \eqref{DiscBound1} is attained if and only if $f(x)=G_{a,\lambda_0}(x)$.

Finally, for any $\lambda, a \in \C$, where $\lambda \notin \{2\lceil d/2\rceil-1,
2\lceil d/2\rceil+1, \dots, 2d-3\}$, we have
\begin{align}\label{genfff222}
 \Delta_{G_{a,\lambda}}=
 a^{d(d-1)}  \frac{\prod_{k=1}^d k^k \prod_{k=1}^{\lfloor d/2\rfloor-1}
 (\lambda-2k)^{2k}} {\prod_{k=\lceil d/2\rceil}^{d-1}
 (\lambda-2k+1)^{2k-1}}.
\end{align}
\end{theorem}

Note that the left hand side of \eqref{genfff2} is decreasing in $\lambda_0$ from $\infty$ to $1$ when $\lambda_0 \in (2d-3,\infty)$. In the only case $m=2^{d-1}a^d$ that is allowed in both Theorems~\ref{Vfixed} and \ref{otherpart}, by the identity
\begin{equation}\label{sdr}
1+ \sum_{k=1}^{\lfloor d/2 \rfloor} {d\choose 2k} =2^{d-1}
\end{equation}
and \eqref{genfff2},
we obtain $\lambda_0=\lambda_0(a,d,2^{d-1}a^d)=2d-2$. Hence $\lambda_0=\lambda_0(a,d,m) \in [2d-2,\infty)$ is indeed unique for each $m$ in the range \eqref{Vucond}, which corresponds to the range $(1,2^{d-1}]$ for the right hand side of \eqref{genfff2}.
Moreover, by \eqref{aa102} with $B=0$ and \eqref{genfff} with $\lambda_0=2d-2$, we obtain
\begin{equation}\label{lopas}
F_{a,0}(x)=G_{a,2d-2}(x)=\frac{(x+ai)^d+(x-ai)^d}{2}
=x^{d}+\sum_{k=1}^{\lfloor d/2 \rfloor} (-1)^k a^{2k} {d \choose 2k} x^{d-2k}.
\end{equation}

Next, we give a completely explicit version of Theorem~\ref{otherpart} with $a=1$ for $d = 2,3,4,5$.

\begin{cor}\label{smalld}
Let $a=1$ in Problem~\ref{prob2}.
For $d=2$ and $1<m \leq 2$, the maximal value for $\Delta$ is $4(m-1)$.
It is attained iff $\{x_{1}, x_2\}= \{-\sqrt{m-1},\sqrt{m-1}\}$.

For $d=3$ and $1<m \leq 4$, the maximal value for $\Delta$ is $4(m-1)^3$.
It is attained iff $\{x_1,x_2,x_3\}=\{-\sqrt{m-1},0,\sqrt{m-1}\}$.

For $d=4$ and $1<m \leq 8$, the maximal value for $\Delta$ is
\begin{equation}\label{large}
\frac{1024}{3125} \Big(2(m^2+7m+1)^{3/2}(m^2-18m+1)+(m+1)\big(2m^4-17m^3+462m^2-17m+2\big)\Big).
\end{equation}
It is attained iff $x_1,x_2,x_3,x_4$ are the roots of the polynomial
\begin{equation}\label{ff}
G(x)=x^4-\frac{2}{5} \big(m-4+\sqrt{m^2+7m+1}\big)x^2 +
\frac{3m+3-2\sqrt{m^2+7m+1}}{5}.
\end{equation}

For $d=5$ and $1<m \leq 16$, the maximal value for $\Delta$ is
\begin{align}\label{large22}
\Delta_G &=\frac{55296}{823543} (m^2+23m+1)^{3/2}(m^2-54m-25)\Big(m^2+2m-\frac{1}{27}\Big) \\&
+\frac{55296}{823543} (m+1)\Big(m^6-\frac{37}{2} m^5+
\frac{56755}{27} m^4+\frac{287435}{27}m^3+\frac{56755}{27} m^2-\frac{37}{2} m+1\Big). \nonumber
\end{align}
It is attained iff $x_1,x_2,x_3,x_4,x_5$ are the roots of the polynomial
\begin{equation}\label{ff22}
G(x)=x^5-\frac{2}{7} \big(m-6+\sqrt{m^2+23m+1}\big)x^3 +
\frac{5m+5-2\sqrt{m^2+23m+1}}{7} x.
\end{equation}
\end{cor}

The expressions \eqref{large} and \eqref{large22} indicate that one should not expect an explicit version of Theorem~\ref{BoundD}
similar to \eqref{aa101} when the inequality opposite to \eqref{aa100} holds. Although we do find an explicit monic polynomial that realizes the minimum (as described in Theorem~\ref{otherpart}) and know its discriminant by \eqref{genfff222}, it is impossible
to express  the smallest value for $m$ in terms of $a,d,D$
explicitly already for small values of $d$, say $d=4$ (see \eqref{large}) and $d=5$ (see \eqref{large22}).
Note that $\lambda_0$ is a root of the polynomial of degree
$\lfloor d/2 \rfloor$ by \eqref{genfff2}, so it is impossible to find
$\lambda_0$ explicitly in terms of $a,d,m$
for $d \geq 10$.

The rest of our paper is organized as follows. In the next section, we give two applications of the main results to problems on the largest disks contained in lemniscates,
and to the minimum energy problems for discrete charges on the real line.

Section \ref{Lagrange} deals with Lagrange multiplier approach to Problem \ref{prob2},
and its relation to Problem \ref{prob1}. In particular, we prove that the extremal polynomials for Problem \ref{prob2} satisfy a second order differential equation,
which implies that their coefficients satisfy certain recurrence relations; see Theorem \ref{extremals}. The mentioned recurrences allow us to find two different families
of extremal polynomials in Theorem \ref{generic}, and analyse the complete range of possible Lagrange multipliers corresponding to these families. Results of Section \ref{Lagrange} serve as the main ingredients of the proof of Theorem \ref{otherpart}.
It turns out that only polynomials from family \eqref{genf} given in Theorem~\ref{generic} can attain the value for $m$ as in \eqref{Vucond}.

On the other hand,
Theorem~\ref{BoundD} will be proved directly, without the use of the extremal families obtained in Theorem~\ref{generic}. This time, in the
range for $m$ as in \eqref{Vcond} there are polynomials in both families that attain this $m$. By Theorem~\ref{Vfixed}, for each
$m$ satisfying \eqref{Vcond} the polynomials from \eqref{genf2} have smaller discriminants than those from \eqref{genf}.
For the proof of Theorem~\ref{BoundD} we relate our constrained extremal problem on the real line to a well known problem
of maximizing the absolute value of the discriminant for points on the unit circle. For this, in Section~\ref{Aux} we state and prove some auxiliary results on
maxima of the arising products of cosines and sines. In Section~\ref{Aux22} we discuss the relation between the polynomials that appear in
\eqref{genfff} and Gegenbauer (or ultraspherical) polynomials
\begin{equation}\label{gegen}
C_d^{\mu}(x):=\sum_{k=0}^{\lfloor d/2 \rfloor}(-1)^k \frac{\mu(\mu+1)\dots(\mu+d-k-1)}{k! (d-2k)!} (2x)^{d-2k},
\end{equation}
which are special cases of Jacobi polynomials.

Finally, we will collect all proofs of the main results in Section \ref{Proofs} and group them by section.

\section{Applications} \label{Apps}

Let us consider some applications of the extremal problems from the previous section to questions about the size of lemniscates for polynomials with real zeros, and about the discrete minimum energy configurations of charges on the real line.

For a given polynomial $f$, let $E(f)$ be the filled-in lemniscate $\{z \in \C \>:\> |f(z)| \leq 1\}$. Studies of geometric structure, shape and size of lemniscates are classical in many areas of mathematics. Lemniscates play important roles in various problems of analysis, algebraic geometry, number theory, applied mathematics, etc. Many interesting problems about polynomial lemniscates originated in the paper of Erd\H{o}s, Herzog and Piranian \cite{EHP}, and some of them still remain open. The latter paper considered problems related to the size and shape of lemniscates for polynomials with zeros in the unit disk, and with real zeros. In particular, in \cite{EHP} it is shown that there is a sequence of monic polynomials $f_d$ of degree $d\to\infty$, with all zeros in the closed unit disk, such that the areas of $E(f_d)$ decay to zero as $d\to\infty$. The authors also asked a number of questions related to the rate of this decay and the size of the largest disk contained in $E(f)$. Erd\H{o}s and Netanyahu \cite{EN} proved that if all the roots of a monic polynomial $f$ of degree $d$ are contained in a fixed compact connected set of transfinite diameter (logarithmic capacity) $c<1,$ then $E(f)$ contains a disk of radius $r_c$ that depends only on $c$. The assumptions of Erd\H{o}s and Netanyahu imply that $\limsup_{d\to\infty} \Delta_f^{1/(d(d-1))} \le c < 1$, i.e., the roots are relatively close to each other, while in the case of the unit disk the value of this $\limsup$ can be equal to $1$, indicating much better separation of roots.

More details on transfinite diameter and capacity can be found in many books on potential theory, see \cite{Ran}, for example. Erd\H{o}s \cite{Erd} conjectured that for \emph{any} set of transfinite diameter $1$ there is a sequence of monic polynomials $f_d$ of degree $d\to\infty$, with all zeros contained in this set, such that the areas of $E(f_d)$ decay to zero as $d\to\infty$, and hence $E(f_d)$ cannot contain disks of a fixed radius. This conjecture remains open. We show, however, that one can construct a sequence of monic polynomials $f_d$ of degree $d$ with real zeros that are well separated in the sense that the discriminant of $f_d$ is as large as $2^{1-d}\,d^d$, and $E(f_d)$ contains a disk of radius $2^{-1+1/d}>1/2$. This result comes from the following consequence of Theorem \ref{BoundD}
with $D=2^{1-d} d^d$.

\begin{cor}\label{Bnonzero11}
Assume that $D>0$, $d \geq 2$ and
\begin{equation}\label{aa1}
a_0:=2^{-1+2/d} d^{-1/(d-1)} D^{1/d(d-1)}.
\end{equation}
Then, there is a unique $f \in K(d,D)$, namely,
\begin{equation}\label{aa0.5}
f(x)=\frac{(x+a_0i)^d+(x-a_0i)^d}{2},
\end{equation}
satisfying
\begin{equation}\label{aa0}
|f(a_0i)|=2d^{-d/(d-1)}D^{1/(d-1)}.
\end{equation}
\end{cor}

Note that the lemniscate $E(f)$ is symmetric with respect to the real line for any monic polynomial $f$ with real roots. Pommerenke \cite{Pom1} showed that
$E(f)$ for such $f$ is a union of closed disks centered on the real line, and the diameter of the largest disk contained in $E(f)$ is equal to the
vertical width of the set $E(f)$.

Let $r(d,D)$ be the largest possible radius of the disk with center on the real line that is contained in the set $E(f)$, among all $f \in K(d,D)$. Evidently,
$r(d,D) < 1$ for $d \geq 2$, $D>0$, and $\lim_{D \to 0} r(d,D)=1$ for each $d \geq 2$. Also, one can see that $r(d,D)$ as a function in $D \in [0,\infty)$
is decreasing from $1$ to $0$. In the next theorem, we describe $r(d,D)$ in terms of Problem~\ref{prob1}, and investigate the situation when $D$
is not too small.

\begin{theorem}\label{FatLem}
Let $d \geq 2$ and $D>0$. Then $r(d,D)$ is the largest $r$
for which $|f(ri)|=1$ holds for some $f \in K(d,D)$. Furthermore,
\begin{equation}\label{gneral}
r(d,D) \leq \frac{d^{1/(d-1)}}{2 D^{1/d(d-1)}},
\end{equation}
and
\begin{itemize}
\item[$(i)$] If $d \geq 2$ and $1 \le D\le 2^{1-d}\, d^d$, then
 $2^{-1+2/d} d^{-1/(d-1)} \leq r(d,D) \leq d^{1/(d-1)}/2$;

\item[$(ii)$] If $\lim_{d \to \infty} |\log D|/d^2 = 0$, then $\lim_{d \to \infty}r(d,D) =1/2$;

\item[$(iii)$]  If $\lim_{d\to\infty} (\log D)/d^2 = \infty$, then
$\lim_{d\to\infty} r(d,D) = 0.$
\end{itemize}
\end{theorem}

It follows that the radius of the largest possible disk contained in $E(f)$
for $f \in K(d,D)$
can be close to $1/2$ for large $d$, when the discriminant $D$ is neither too large nor too small, and this constant $1/2$ is best possible. Theorem \ref{FatLem} extends the results of Pommerenke \cite{Pom1} and
\cite{Pom2}, see Theorems 2 and 3 of \cite{Pom2} in particular.

\bigskip
The second application of our results is related to the minimum energy configurations of discrete charges on the real line. It is clear from the definition of
discrete energy below that minimizing this energy (finding the equilibrium position of charges) is equivalent to maximizing the discriminant.
Problems on the equilibrium position of charges on the real line were considered by Stieltjes \cite{Sti1}-\cite{Sti3}, Schur \cite{Sch},
Ismail \cite{Ism} and others. For a general compact set $E$ in the complex plane, points $\{z_k\}_{k=1}^d\subset E$ maximizing the absolute value
of discriminant $\prod_{1\le j<k\le d} |z_j-z_k|^2$ were introduced by Fekete \cite{Fe} in connection with the transfinite diameter of $E$,
and thus are often called Fekete points.  These points are
useful in analysis and computations, e.g., for interpolation of functions, but they are difficult to find explicitly. Fekete points are known only for several
sets such as segment and disk. In particular, the case of $[-1,1]$ was settled by Stieltjes in \cite{Sti1}, while further progress was rather limited.
For example, if the set consists of two intervals of the real line, then Fekete points are not known even for any special configuration.
It is therefore of interest that we are able to find a completely explicit solution of a constrained minimum discrete energy problem described below.
The topic of minimizing discrete energy received close attention in recent years; see, for instance, the book of Borodachov,  Hardin and Saff \cite{BHS}
for the references on this subject.

For a monic polynomial $f$ with real roots $x_1,\dots,x_d$, we consider the associated counting measure
\[
\tau_d = \frac{1}{d} \sum_{k=1}^{d} \delta_{x_k},
\]
where $\delta_x$ denotes the unit point mass at $x$. The logarithmic potential of $\tau_d$ is defined by
\[
U^{\tau_d}(x) = - \int \log|x-t|\,d\tau_d(t) = - \frac{1}{d} \log|f(x)|,
\]
and the discrete energy of $\tau_d$ is defined by
\[
I[\tau_d] = - \frac{1}{d(d-1)} \sum_{j\neq k} \log|x_j-x_k| = - \frac{1}{d(d-1)} \log|\Delta_f|.
\]
Thus it is immediate to see that our Problem \ref{prob2} is equivalent to minimization of the discrete energy $I[\tau_d]$
under the condition $U^{\tau_d}(ai) = - (\log m)/d$, i.e., to finding the equilibrium position of $d$ unit charges on
the real line so that their total potential has a prescribed value at a point off the real line. Recall from the discussion after
the statement of Problem \ref{prob2} that $m\ge a^d$ with equality iff $x_k=0$ for $k=1,\dots,d$. This gives the possible range
for the values $v=U^{\tau_d}(ai) \le - \log a$, with $v=-\log a$ iff $x_k=0,\ k=1,\ldots,d,$ and $I[\tau_d] = \infty.$ Hence
we can only consider the range $v=U^{\tau_d}(ai) < - \log a$ below.

Applying Theorems \ref{Vfixed} and \ref{otherpart}, we obtain the complete description of minimum energy configurations in this setting.

\begin{theorem} \label{Equil}
Suppose that $a>0$, $d \geq 2$ and $v < - \log a$. Then any configuration of points $\{x_k\}_{k=1}^d \subset \R$
such that $U^{\tau_d}(ai) = v$ satisfies
\begin{equation}\label{EnBound}
I[\tau_d] \geq 2v + \log(2a) - \frac{\log d}{d-1}.
\end{equation}

If
\begin{equation}\label{v1}
v \le \left(\frac{1}{d} - 1\right)\log 2 - \log a,
\end{equation}
then equality in \eqref{EnBound} is attained if and only if $\{x_k\}_{k=1}^d$ are either given by \eqref{aa103}, or
by the reflection of points \eqref{aa103} with respect to the origin.

If
\begin{equation}\label{v2}
\left(\frac{1}{d} - 1\right)\log 2 - \log a < v < - \log a
\end{equation}
then
\begin{equation}\label{EnBound2}
I[\tau_d] \geq - \frac{\log \Delta_G}{d(d-1)},
\end{equation}
where $G=G_{a,\lambda_e}$ is the polynomial defined in \eqref{genfff} with  a
unique $\lambda_e=\lambda_e(a,d,v) > 2d-2$ satisfying
\begin{align}\label{Leq}
 1+\sum_{k=1}^{\lfloor d/2 \rfloor} {d\choose 2k} \frac{(2k-1)!!}{\prod_{j=1}^{k} (\lambda_e-2d+2j+1)} =\frac{e^{-vd}}{a^d}.
\end{align}
Moreover, equality in \eqref{EnBound2} is attained if and only if $\{x_k\}_{k=1}^d$ are the roots of $G_{a,\lambda_e}(x)$.
\end{theorem}

As an application of the first part of Theorem \ref{Equil}, we show that the weak* limit for the counting measures of the minimum energy
points, when the value of potential satisfies \eqref{v1}, is given by the arctan distribution. Recall that the weak* convergence
$\tau_d  \stackrel{*}{\rightarrow} \mu$ means that for any continuous $\phi:\R\to\R$ with compact support we have
$\lim_{d\to\infty} \int \phi\,d\tau_d = \int \phi\,d\mu.$

\begin{cor} \label{Lim}
If $a>0$, $d \geq 2$ and $v$ satisfies \eqref{v1}, then the minimum energy points satisfy
\begin{equation}\label{wlim}
\tau_d  \stackrel{*}{\rightarrow} \frac{a\,dx}{a^2+x^2}\quad \mbox{as }d\to\infty.
\end{equation}
\end{cor}

It would be interesting to determine the asymptotic distribution of charges for the remaining range of $v$ given in \eqref{v2}.

\section{Extremal polynomials via Lagrange multipliers} \label{Lagrange}

In this section, we address Problem~\ref{prob2} by the method of Lagrange multipliers. We consider
the equivalent logarithmic version of the maximization problem for
\[
g(x_1,\ldots,x_d)= \log \Delta = \sum_{1\le j<k\le d} \log(x_j-x_k)^2
\]
under the condition
\[
h(x_1,\ldots,x_d)=\log m =  \frac{1}{2}\sum_{k=1}^{d} \log(a^2+x_k^2),
\]
where $x_1,\dots,x_d$ are distinct real numbers. This gives the standard Lagrange multiplier equation $\nabla g = \lambda \nabla h$, where $\lambda\in\R,\ \lambda\neq 0$.
It is clear that Problem~\ref{prob1} is equivalent to minimizing $h$ under the condition $g=\log D,$ which leads to the the Lagrange multiplier equation $\nabla h = \mu \nabla g.$ Thus arising equations are identical by setting $\mu = 1/\lambda$, and the results of this section can be applied to Problem~\ref{prob1} as well.

Replacing each root $x_k$ of the polynomial $f$ by $ax_k$, we change its discriminant from
$\Delta_f$ to $a^{d(d-1)}\Delta_f$, and its value at $ai$ changes to $|f(ai)|=a^d \prod_{k=1}^d \sqrt{1+x_k^2}$.
Thus we can study the normalized version of Problem~\ref{prob2} with $a=1$, restated as follows:

\begin{problem} \label{prob2m}
Given $m>1$ and $d \geq 2$, find all collections of real numbers $x_1,\dots, x_d$ satisfying
$$m=\prod_{k=1}^d \sqrt{1+x_k^2}$$
that give the maximal value of $\Delta=\prod_{1 \leq j<k \leq d} (x_j-x_k)^2$.
\end{problem}

It is clear from the constraint condition that any solution of Problem~\ref{prob2m} must satisfy $|x_k|\le m,\ k=1,\ldots,d.$ Hence we seek the maximum
of continuous function $\Delta$ over a compact set, and a solution of this problem definitely exists. Below, we find a unique critical point for Problem~\ref{prob2m},
by the method of Lagrange multipliers, such that all $x_k,\ k=1,\ldots,d,$ are distinct. Since the minimum of $\Delta$ is equal to zero when some of the points coincide, this critical point provides the maximum of $\Delta$ for Problem~\ref{prob2m}.

We first show that the polynomials $\prod_{k=1}^d (x-x_d)$ corresponding to all possible solutions $x_1,\dots,x_d$ of Problem~\ref{prob2m} satisfy certain second order differential equations,
and hence their coefficients satisfy some useful recurrence relations.

\begin{theorem} \label{extremals}
If
$$f(x)=x^d+c_{d-1}x^{d-1}+\dots+c_0=(x-x_1) \dots (x-x_d)$$ is a solution to Problem \ref{prob2m}, then its coefficients satisfy the equations
\begin{equation}\label{coef}
  (\lambda-2d+2) c_{d-1} = 0
\end{equation}
and
\begin{equation}\label{coef1}
  (k+1)(k+2) c_{k+2} = (d-k)(d+k-1 -  \lambda) c_k,\quad k=0,\ldots,d-2,
\end{equation}
where $\lambda\in\R.$
\end{theorem}

\begin{proof}
We apply the approach of Schur \cite{Sch} and \cite[Section 6.7]{Sze} to Problem \ref{prob2m}.
Consider the equivalent logarithmic version of the maximization problem for
\[
g(x_1,\ldots,x_d)= \log \Delta = \sum_{1\le j<k\le d} \log(x_j-x_k)^2
\]
under the condition
\[
h(x_1,\ldots,x_d)=\log m =  \frac{1}{2}\sum_{k=1}^{d} \log(1+x_k^2),
\]
where $x_1,\dots,x_d$ are distinct real numbers. Clearly,
the standard Lagrange multiplier equation $\nabla g = \lambda \nabla h$, with $\lambda\in\R$, gives
\[
\sum_{j\neq k} \frac{2}{x_k-x_j} - \frac{\lambda x_k}{x_k^2+1}= 0, \quad k=1,\ldots,d.
\]
The latter can be written in the form
\[
 \frac{f''(x_k)}{f'(x_k)} -\frac{\lambda x_k}{x_k^2+1} = 0, \quad k=1,\ldots,d,
\]
or, equivalently,
\[
(x_k^2+1)f''(x_k) - \lambda x_kf'(x_k) = 0, \quad k=1,\ldots,d.
\]
Since $(x^2+1)f''(x) - \lambda xf'(x)$ is a polynomial of degree $d$ that vanishes at $d$ distinct points $\{x_k\}_{k=1}^d$, it must be a constant multiple of $f(x).$
Thus, we arrive at the differential equation
\[
(x^2+1)f''(x) - \lambda xf'(x) = cf(x),
\]
where $c\in\R.$ Equating the leading coefficients of polynomials on both sides gives
\[
d(d-1) - \lambda d= c.
\]
Thus, the differential equation for $f$ takes the form
\begin{align}\label{diffeq}
(x^2+1)f''(x) - \lambda xf'(x) + d(\lambda- d+1)f(x) = 0.
\end{align}
Substituting $f(x) = \sum_{k=0}^{d} c_k x^k$, where $c_d=1$, into \eqref{diffeq}, we obtain
\begin{align*}
 \sum_{k=0}^{d} k(k-1) c_k x^k + \sum_{k=0}^{d} k(k-1) c_k x^{k-2} - \lambda \sum_{k=0}^{d} k c_k x^k +d (\lambda -d+1) \sum_{k=0}^{d} c_k x^k = 0.
\end{align*}

By considering the coefficient for $x^{d-1}$, we get
$$(d-1)(d-2)c_{d-1}-\lambda (d-1)c_{d-1}+d(\lambda -d+1)c_{d-1}= (\lambda-2d+2) c_{d-1} = 0,$$ which is \eqref{coef}. Note also
that, by changing $k$ to $k+2$, we can rewrite
the second sum on the left in the form $\sum_{k=0}^{d-2} (k+2)(k+1) c_{k+2} x^k$. Evaluating coefficients for $x^k$, $k=0,1,\dots,d-2$, we find that
$$k(k-1)c_k+(k+2)(k+1)c_{k+2}-\lambda k c_k +d(\lambda-d+1)
c_k=0.$$
By the identity
$$k(k-1)-\lambda k+d(\lambda-d+1)=(d-k)(\lambda-d-k+1),$$
this leads to
$$
(k+1)(k+2) c_{k+2} +(d-k)(\lambda-d-k +1) c_k = 0
$$
for each $k=0,\dots,d-2$, as stated in \eqref{coef1}.
\end{proof}

The next theorem describes various $f$ for all possible $\lambda$.

\begin{theorem} \label{generic}
If the Lagrange multiplier $\lambda \neq 2d-2$, then $\lambda>2d-3$ and the solution of Problem \ref{prob2m} is
contained in the family of polynomials
\begin{align}\label{genf}
  f(x) = x^d + \sum_{k=1}^{\lfloor d/2 \rfloor} \left({d\choose 2k} \prod_{j=1}^{k} \frac{2j-1}{2d-2j-1-\lambda} \right) x^{d-2k}.
\end{align}
Here, the values of $\lambda \neq 2d-2$ corresponding to extremal polynomials \eqref{genf} must satisfy the constraint equation $|f(i)|=m$.

If $\lambda=2d-2$, then
\begin{equation}\label{genf2}
  f(x) = \frac{1}{2d} \Big( (d-Bi)(x+i)^d + (d+Bi)(x-i)^d \Big),
\end{equation}
where $B=c_{d-1}$ may be found from $|f(i)|=m$.
\end{theorem}

\begin{proof}
We begin with the case $\lambda=2d-2$. Inserting this value of $\lambda$ into \eqref{coef1}, we deduce that
for $k=0,\ldots,d-2$
$$
    c_k = \frac{(k+1)(k+2)}{(d-k)(d+k-1 -  \lambda)} c_{k+2} =
     - \frac{(k+1)(k+2)}{(d-k)(d-k-1)} c_{k+2}=- \frac{{d\choose k}}{{d\choose k+2}} c_{k+2}.
$$
The last recursion, used with initial values $c_d=1$ and $c_{d-1}=B$, implies that
\begin{align*}
    c_{d-2k} = (-1)^k {d\choose 2k},\quad k=0,\ldots,\lfloor d/2 \rfloor,
\end{align*}
and
\begin{align*}
    c_{d-2k-1} = \frac{(-1)^k B}{d} {d\choose 2k+1},\quad k=0,\ldots,\lfloor (d-1)/2 \rfloor.
\end{align*}
Hence the extremal polynomial $f$ takes the form
\begin{align*}
  f(x) &= \sum_{k=0}^{\lfloor d/2 \rfloor} (-1)^k  {d\choose 2k} x^{d-2k} + \sum_{k=0}^{\lfloor (d-1)/2 \rfloor} \frac{(-1)^k B}{d} {d\choose 2k+1} x^{d-2k-1} \\
  &= \frac{1}{2}\Big((x+i)^d + (x-i)^d \Big) + \frac{B}{2di}\Big((x+i)^d - (x-i)^d \Big) \\
  &= \frac{1}{2d} \Big( (d-Bi)(x+i)^d + (d+Bi)(x-i)^d \Big),
\end{align*}
which is \eqref{genf2}.

Assume now that $\lambda \ne d+k-1$ for $k=0,1,\dots,d-1$.
Since $\lambda \ne 2d-2$,  we find that $c_{d-1} = 0$ by \eqref{coef}. Also, from \eqref{coef1} and $\lambda \ne d+k-1$ for $k=0,\dots,d-2$, it follows that
\begin{align*}
    c_k = \frac{(k+1)(k+2)}{(d-k)(d+k-1-\lambda)} c_{k+2}.
\end{align*}
Applying the latter relation iteratively, with initial value $c_{d-1}=0$, one can easily see that $c_{d-2k-1}=0$
for $k=0,\ldots,\lfloor (d-1)/2 \rfloor$. Similarly, applying it with initial value $c_d=1$, we find that
\begin{equation}\label{ckk}
    c_{d-2k} =
    {d\choose 2k} \prod_{j=1}^{k} \frac{2j-1}{2d-2j-1-\lambda}
    = (-1)^k {d\choose 2k} \frac{(2k-1)!!}{\prod_{j=1}^k (\lambda-2d+2j+1)}
\end{equation}
for $k=1,\ldots,\lfloor d/2 \rfloor$. This implies \eqref{genf}.
Since the polynomial \eqref{genf} is of the form $g(x^2)$ (if $d$ is even) or $xg(x^2)$ (if $d$ is odd), it may only have $d$ real roots if $c_{d-2}<0$ by Descartes' rule of signs. This yields $\lambda>2d-3$ by \eqref{ckk} with $k=1$.

Finally, suppose that $\lambda = d+K-1$ for some $K \in \Z$  in the range $0\le K \le d-2$.  Then $c_{d-1}=0$ by \eqref{coef}. From \eqref{coef1} we find that
\[
    (k+1)(k+2) c_{k+2} = (d-k)(k-K) c_k,\quad k=0,\ldots,d-2.
\]
Inserting $k=K,K+2,\dots$, we see that $c_{K+2j}=0$ for all $j=1,2,\ldots,$ such that $K+2j\le d.$ Since $c_d=1$, we get $K+2j\neq d$, and so $d-K$ must be odd.
Iterating the above recurrence relation from $c_K$ to find the lower coefficients, we obtain that
\begin{equation} \label{Kcoef}
  c_{K-2k} = (-1)^k {K\choose 2k} \prod_{j=1}^{k} \frac{2j-1}{d-K+2j} c_K,\quad k=1,\ldots,\lfloor K/2 \rfloor.
\end{equation}
Similarly, starting with $c_d=1$, and iterating as in the first part of the proof, we derive that
\begin{align} \label{dcoef}
    c_{d-2k} = {d\choose 2k} \prod_{j=1}^{k} \frac{2j-1}{d-K-2j},\quad k=1,\ldots,\lfloor d/2 \rfloor,
\end{align}
which yields
\begin{align}\label{degenf}
  f(x) &= x^d + \sum_{k=1}^{\lfloor d/2 \rfloor} \left({d\choose 2k} \prod_{j=1}^{k} \frac{2j-1}{d-K-2j} \right) x^{d-2k} \\
  &+ c_K x^K + c_K \sum_{k=1}^{\lfloor K/2 \rfloor} \left( (-1)^k {K\choose 2k} \prod_{j=1}^{k} \frac{2j-1}{d-K+2j} \right) x^{K-2k}.\nonumber
\end{align}

Now, by using Descartes' rule of signs, we will show that the above polynomials \eqref{degenf} cannot have $d$ real roots,
so that this family is not compatible with the assumptions of Problem~\ref{prob2m}. It is obvious that the coefficients of \eqref{degenf} listed
in \eqref{Kcoef} alternate in sign, unless $c_K=0$ and then all of them vanish. The coefficients given in \eqref{dcoef} are positive if $2k<d-K$,
and alternate in sign for larger $k$. This means that in the list of all coefficients for \eqref{degenf}, arranged in the decreasing order of index, we first have $d-K$ nonnegative coefficients
(from $x^d$ to $x^{K+1}$). Then,
we have coefficients \eqref{Kcoef} and \eqref{dcoef} interlaced
from $x^K$ to $x^0$.

If $c_K \geq 0$ then the coefficients for $x^K, x^{K-1}, x^{K-2}, x^{K-3}$ have signs $+--+$ and periodically afterwards. This gives at most $[(K+1)/2]$ sign changes for $f$. In case $c_K<0$ the corresponding signs are $--++$ and continue periodically afterwards.
This gives at most $[(K+2)/2]$ sign changes for $f$.  Consider the polynomial $(-1)^d f(-x)$. Then, the picture is the same except that $c_K$ becomes $-c_K$, since $d+K=d-K+2K$ is odd. So the number of both positive and negative roots of $f$ cannot exceed
$[(K+2)/2]+[(K+1)/2]=K+1$.
Adding one more possible
root $x=0$ we obtain at most $K+2$ real roots for $f$.
Since $d-K$ is odd, $K \ne d-2$.
Consequently, $K \leq d-3$, and so $f$ has at most
$K+2 \leq d-1$
real roots. This proves that not all the roots of $f$ given in \eqref{degenf} are real.
\end{proof}

\section{Products of sines and cosines} \label{Aux} 

We give some explicit values and estimate for products of sines and cosines arising in the proofs of main results.

\begin{lemma}\label{prodcos}
For each integer $d \geq 2$, we have
$$
P(x):=\prod_{k=0}^{d-1} \cos^2\Big(x+\frac{\pi k}{d}\Big) = 2^{2-2d} \sin^2(dx-d\pi/2) = \left\{
                                                                                           \begin{array}{ll}
                                                                                             2^{2-2d} \cos^2 (dx), & \hbox{if $d$ is odd;} \\
                                                                                             2^{2-2d} \sin^2 (dx), & \hbox{if $d$ is even.}
                                                                                           \end{array}
                                                                                         \right.
$$
\end{lemma}

\begin{proof}
We use the identity
\begin{equation}\label{qq1}
\sin (dx) = 2^{d-1} \prod_{k=0}^{d-1} \sin \Big(x+\frac{\pi k}{d}\Big)
\end{equation}
found in \textbf{1.392} of \cite[p. 41]{GR}. This immediately gives the corresponding formula for $P(x)$
\begin{align*}
\prod_{k=0}^{d-1} \cos^2\Big(x+\frac{\pi k}{d}\Big) &= \prod_{k=0}^{d-1} \sin^2\Big(x-\frac{\pi}{2}+\frac{\pi k}{d}\Big) \\
                                                                                   &= 2^{2-2d} \sin^2(dx-d\pi/2) = \left\{
                                                                                           \begin{array}{ll}
                                                                                             2^{2-2d} \cos^2 (dx), & \hbox{if $d$ is odd;} \\
                                                                                             2^{2-2d} \sin^2 (dx), & \hbox{if $d$ is even,}
                                                                                           \end{array}
                                                                                         \right.
\end{align*}
and thus completes the proof of the lemma.
\end{proof}

\begin{lemma}\label{maxcos}
For any $y_1,\dots,y_d \in [0,\pi)$, we have
\begin{equation}\label{cg2}
\prod_{1 \leq j<k \leq d} \sin^2 (y_j-y_k) \leq 2^{-d(d-1)} d^d.
\end{equation}
Furthermore, equality in \eqref{cg2} is attained if and only if the set $\{y_1,\dots,y_d\}$ is an arithmetic progression
with difference $\pi/d$.
\end{lemma}

\begin{proof}
By subtracting $y:=\min_{1 \leq i \leq d} y_i$ from each $y_k$, $k=1,\dots,d$, and then rearranging the new elements $y_k-y$
in ascending order,
we may assume that
$y_1=0 < y_2 < \dots < y_d <\pi$.
Notice that $$2\sin (y_k-y_j) =|e^{2i y_k} -e^{2i y_j}|$$
for any pair of indices $j<k$ satisfying $1 \leq j<k \leq d$. Hence
$$2^{d(d-1)} \prod_{1 \leq j<k \leq d} \sin^2 (y_j-y_k) =\prod_{1 \leq j<k \leq d} |e^{2iy_k}-e^{2i y_j}|^2.$$
Here, the product on the right hand side is the square of the absolute value of the Vandermonde determinant for $e^{2i y_1}=1, e^{2i y_2}, \dots, e^{2i y_d}$.
It is well known that the maximum of the latter does not exceed $d^d$, with equality iff $y_k=\pi (k-1)/d$ for $k=2,\dots,d,$ by Hadamard's inequality, cf. \cite{BC}.
See also \cite{Ger} for an alternative proof of this fact due to Fekete. This implies the assertion of the lemma.
\end{proof}

\section{Polynomials $G_{a,\lambda}$ in terms of Jacobi and Gegenbauer polynomials} \label{Aux22}

Note that the polynomial \eqref{genfff} is defined for any $a \in \C$
and any $\lambda \in \C$, where $\lambda \ne 2d-2j-1$ for $j \in \{1,2,\dots,\lfloor d/2 \rfloor\}$. It is easy to see that the latter condition is equivalent to
\begin{equation}\label{bubu}
\lambda \notin \{2\lceil d/2\rceil-1,
2\lceil d/2\rceil+1, \dots, 2d-3\},
\end{equation}
where the right hand side of \eqref{genfff222} is defined. The formula
\eqref{genfff222} obviously holds for $a=0$, so from now on we assume that $a \ne 0$. In all what follows we will first prove \eqref{genfff222} for
all real $\lambda$ greater than $2d-2$ and then give an argument which extends this formula to all complex $\lambda$
satisfying \eqref{bubu}.

Recall that Jacobi polynomials are defined by
\begin{align}\label{Jacobi}
  P_d^{(\alpha,\beta)}(x) &:= 2^{-d} \sum_{k=0}^d {d+\alpha \choose d-k} {d+\beta \choose k} (x-1)^k (x+1)^{d-k} \\
  &= \frac{(\alpha + \beta + d + 1)_d}{d!\, 2^d} x^d + \ldots, \nonumber
\end{align}
where $(t)_d:=t(t+1)\dots (t+d-1)$ is Pochhammer's symbol (or the rising factorial), and $${t \choose d}:=\frac{t(t-1)\dots(t-d+1)}{d!}$$
is a generalized binomial coefficient. In the special case, when $\alpha=\beta=\mu-1/2$, Jacobi polynomials \eqref{Jacobi} are also expressible
as
\begin{equation}\label{ppcc}
P_d^{(\mu-1/2,\mu-1/2)}(x)=\frac{(\mu+1/2)_d}{(2\mu)_d} C_d^{\mu}(x),
\end{equation}
 where
$C_d^{\mu}(x)$ is defined in \eqref{gegen}, see (4.5.1) of \cite[p. 94]{Ismbook}.

Let us evaluate the polynomial \eqref{genfff} at $iax$. We have
$$(ia)^{-d} G_{a,x}(iax)= x^d + \sum_{k=1}^{\lfloor d/2 \rfloor} \left({d\choose 2k} \frac{(2k-1)!!}{\prod_{j=1}^{k} (\lambda-2d+2j+1)} \right) x^{d-2k}.$$
Notice that
$${d \choose 2k} (2k-1)!!=\frac{d!}{2^k k! (d-2k)!}$$
and
\begin{align*}
\prod_{j=1}^k (\lambda-2d+2j+1) & =2^k (-\mu-d+1)_k=(-1)^k 2^k
(\mu+d-1) \dots (\mu+d-k) \\& =(-1)^k 2^k \frac{(\mu)_d}{(\mu)_{d-k}},
\end{align*}
where $\mu:=-(\lambda+1)/2$. Therefore, using these identities, \eqref{gegen}
and \eqref{ppcc}, we derive that
\begin{align*}
(ia)^{-d} G_{a,\lambda}(iax) & = x^d+\sum_{k=1}^{\lfloor d/2 \rfloor}
(-1)^k\frac{d!}{2^{2k}k!(d-2k)!} \frac{(\mu)_{d-k}}{(\mu)_d} x^{d-2k}
\\& =\sum_{k=0}^{\lfloor d/2 \rfloor} (-1)^k\frac{d!}{2^{2k}k!(d-2k)!} \frac{(\mu)_{d-k}}{(\mu)_d} x^{d-2k} \\& = \frac{2^{-d} d!}{(\mu)_d} \sum_{k=0}^{\lfloor d/2 \rfloor} (-1)^k\frac{(\mu)_{d-k}}{k!(d-2k)!} (2x)^{d-2k} =\frac{2^{-d} d!}{(\mu)_d} C_d^{\mu}(x)
\\&= \frac{d! (2\mu)_d}{2^d (\mu+1/2)_d (\mu)_d} P_d^{(\mu-1/2,\mu-1/2)}(x)=
\frac{2^d d!}{(2\mu+d)_d} P_d^{(\mu-1/2,\mu-1/2)}(x).
\end{align*}
Note that, by \eqref{Jacobi}, the leading coefficient $P_d^{(\mu-1/2,\mu-1/2)}(x)$ equals $(2\mu+d)_d/(2^d d!)$, so the polynomials standing at leftmost and rightmost of this equality are both monic.
In view of $\mu=-(\lambda+1)/2$, this yields
\begin{equation}\label{ident}
G_{a,\lambda}(x)=\frac{(2ai)^d d!}{(-1)^d  (\lambda-2d+2)_d} P_d^{(-\lambda/2-1,-\lambda/2-1)}(-ix/a).
\end{equation}

It is important to observe that the value of Jacobi polynomial parameters $\alpha=\beta=-\lambda/2-1<-d$ in our case, i.e.,
these parameters are outside the classical range $\alpha,\beta>-1$ typically considered in most of references.
Extending the formula for the discriminant of Jacobi polynomials found in (3.4.16) of \cite[p. 69]{Ismbook}
to arbitrary parameters $\alpha$ and $\beta$, in Lemma 5.3 of \cite{DP} we have shown the following:

\begin{lemma}\label{Jdiscr}
Let $P_d^{(\alpha,\beta)}$ be the general Jacobi polynomial defined in \eqref{Jacobi} for $\alpha,\beta\in\C$ and some fixed $d \geq 2$. If
$\alpha+\beta\neq -d-k,\ k=1,\ldots,d,$ then the discriminant of $P_d^{(\alpha,\beta)}$ is given by
\begin{align}\label{Jdis}
  \Delta_{P_d^{(\alpha,\beta)}}= 2^{-d(d-1)}\prod_{k=1}^{d}
  k^{k-2d+2} (k+\alpha)^{k-1} (k+\beta)^{k-1} (d+k+\alpha+\beta)^{d-k}.
\end{align}
\end{lemma}

Let us apply this lemma to $\alpha=\beta=-\lambda/2-1$. Then, the condition on $\alpha+\beta$ is satisfied, because $\lambda>2d-2$.
 By \eqref{Jdis}, we find that
\begin{align*}
\Delta_{P_d^{(-\lambda/2-1,-\lambda/2-1)}}&= 2^{-d(d-1)} \prod_{k=1}^{d}
  k^{k-2d+2} (k-\lambda/2-1)^{2k-2} (d+k-\lambda-2)^{d-k}
  \\&=\frac{(-1)^{d(d-1)/2}}{2^{2d(d-1)}d!^{2d-2}} \prod_{k=1}^{d} k^k (\lambda+2-2k)^{2k-2}(\lambda+2-d-k)^{d-k}.
\end{align*}
 The discriminant of the polynomial $P_d^{(-\lambda/2-1,-\lambda/2-1)}(-ix/a)$ is thus
  the above number multiplied by $(-1)^{d(d-1)/2}a^{-d(d-1)}$, that is,
  $$\frac{1}{a^{d(d-1)} 2^{2d(d-1)}d!^{2d-2}} \prod_{k=1}^{d} k^k (\lambda+2-2k)^{2k-2}(\lambda+2-d-k)^{d-k}.$$
  In order to find the discriminant of $G_{a,\lambda}$, we need to multiply this by $c^{2d-2}$, where $c$ is the constant factor
  near $P_d^{(-\lambda/2-1,-\lambda/2-1)}$ in \eqref{ident}. Since
  $$c^{2d-2}=\frac{(2ai)^{2d(d-1)}d!^{2d-2}}{(\lambda-2d+2)_d^{2d-2}}= a^{2d(d-1)} 2^{2d(d-1)} d!^{2d-2}
  \prod_{k=1}^d (\lambda+2-d-k)^{2-2d},$$
  we deduce that
  \begin{equation}\label{gg22}
  \Delta_{G_{a,\lambda}}= a^{d(d-1)} \prod_{k=1}^{d} \frac{k^k (\lambda+2-2k)^{2k-2}}{(\lambda+2-d-k)^{d+k-2}}
  \end{equation}
 for $\lambda>2d-2$.
  Next,
  we express
  the factors containing $\lambda$ in the nominator of this fraction in the form
  $$\prod_{k=1}^d (\lambda+2-2k)^{2k-2}= \prod_{k=1}^{\lfloor d/2 \rfloor-1} (\lambda-2k)^{2k}  \prod_{k=\lfloor d/2 \rfloor}^{d-1} (\lambda-2k)^{2k}.$$
  Similarly, since each $k \in \{1,\dots,d\}$ can be written either as $k=2j-d+2$ with integer $j$ in the range $\lceil (d-1)/2 \rceil \leq j \leq d-1$ or as $k=2j-d+1$ with integer $j$ satisfying $\lceil d/2 \rceil \leq j \leq d-1$, we can split the factors with $\lambda$ in the denominator of \eqref{gg22} into two parts as follows:
  $$\prod_{k=1}^{d} (\lambda+2-d-k)^{d+k-2}
  = \prod_{j=\lceil (d-1)/2 \rceil}^{d-1} (\lambda-2j)^{2j} \prod_{j=\lceil d/2 \rceil}^{d-1} (\lambda-2j+1)^{2j-1}.$$
  Note that $\lfloor d/2 \rfloor=\lceil (d-1)/2 \rceil$, so the term $\prod_{k=1}^{\lfloor d/2 \rfloor-1} (\lambda-2k)^{2k}$ cancels out, and hence
  \eqref{gg22} implies \eqref{genfff222} for each $\lambda>2d-2$.

As we already observed above, the right hand side of \eqref{genfff222}
is defined for all complex $\lambda$ satisfying \eqref{bubu} exactly when the polynomial $G_{a,\lambda}$ is defined. To extend the formula \eqref{genfff222} from real
$\lambda>2d-2$ to complex $\lambda$ in the range as claimed, we can use the same argument as that in the proof of Lemma~\ref{Jdiscr} (see Lemma 5.3 in \cite{DP}).
Since the discriminant $\Delta_{G_{a,\lambda}}$ is a polynomial in the coefficients of $G_{a,\lambda}$, it is a rational function
in $\lambda$ by \eqref{genfff}. Note that the right hand side of \eqref{genfff222} is also a rational function in
$\lambda$. These two rational functions
coincide for $\lambda>2d-2$. Hence they must coincide for each
$\lambda \in \C \setminus \{2\lceil d/2\rceil-1,
2\lceil d/2\rceil+1, \dots, 2d-3\}$ by the uniqueness theorem for holomorphic functions. This completes the proof of \eqref{genfff222}.

\section{Proofs} \label{Proofs}

\subsection{Proofs for Section \ref{ExPr}}

\begin{proof}[Proof of Theorem~\ref{BoundD}]
Write the roots $x_k$ of $f \in K(d,D)$ in the form $x_k=a \tan y_k$, where
$y_k \in [0,\pi/2) \cup (\pi/2,\pi)$ for $k=1,\dots,d$. Then,
$$|f(ai)|^2 = \prod_{k=1}^d (a^2+a^2 \tan^2 y_k)=a^{2d} \prod_{k=1}^d \frac{1}{\cos^2 y_k}.$$
Also,
\begin{align*}
D &= \Delta_f=\prod_{1 \leq j < k \leq d} (a \tan y_j -a \tan y_k)^2=
a^{d(d-1)} \prod_{1 \leq j < k \leq d} \frac{\sin^2 (y_j-y_k)}{\cos^2 y_j \cos^2 y_k} \\& =a^{d(d-1)} \Big(\prod_{k=1}^d \frac{1}{\cos^2 y_k}\Big)^{d-1} \prod_{1 \leq j<k \leq d} \sin^2 (y_j-y_k).
\end{align*}
This yields
$$
\frac{a^{d(d-1)} D}{|f(ai)|^{2(d-1)}}=
\prod_{1 \leq j<k \leq d} \sin^2 (y_j-y_k).
$$
Bounding the right hand side from above by Lemma~\ref{maxcos}, we
find that
\begin{equation}\label{loi1}
\frac{a^{d(d-1)} D}{|f(ai)|^{2(d-1)}} \leq 2^{-d(d-1)} d^d.
\end{equation}
Now, by rewriting this inequality in the form $|f(ai)|^{2(d-1)} \geq (2a)^{d(d-1)} d^{-d} D$ and taking $2(d-1)$th root of both sides one gets  \eqref{aa101}.

By Lemma~\ref{maxcos}, equality in \eqref{loi1} and so in \eqref{aa101} holds iff $\{y_1,\dots,y_d\} \in [0,\pi/2) \cup (\pi/2,\pi)$ is an arithmetic progression with difference $\pi/d$. Then
$$
\{y_1,\dots,y_d\}=\{\gamma, \gamma+\pi/d, \gamma+2\pi/d,\dots,\gamma+(d-1)\pi/d\}
$$
for some $\gamma \in [0,\pi/d)$, so that
\begin{align*}
\frac{D^{1/d(d-1)}}{a} &=  \Big(\prod_{k=1}^d \frac{1}{\cos^2 y_k}\Big)^{1/d} \Big(\prod_{1 \leq j<k \leq d} \sin^2 (y_j-y_k)\Big)^{1/d(d-1)} =\frac{d^{1/(d-1)}}{2}\Big(\prod_{k=1}^d \frac{1}{\cos^2 y_k}\Big)^{1/d} \\&=
\frac{d^{1/(d-1)}}{2} \Big(\prod_{k=0}^{d-1} \frac{1}{\cos^2 (\gamma+\pi k/d)}\Big)^{1/d} = \frac{d^{1/(d-1)}}{2P(\gamma)^{1/d}},
\end{align*}
with $P(x)$ as defined in Lemma~\ref{prodcos}. Therefore, $P(\gamma) = (a/2)^d d^{d/(d-1)} D^{-1/(d-1)}.$
Now, in view of Lemma~\ref{prodcos} we arrive at the equations
\begin{equation}\label{nbv1}
\cos^2 (d\gamma) = p(a,d,D)^2
\end{equation}
 for odd $d$, or
 \begin{equation}\label{nbv2}
 \sin^2 (d\gamma) = p(a,d,D)^2
 \end{equation}
  for even $d$, where
$$p(a,d,D)=a^{d/2}2^{d/2-1} d^{d/(2d-2)} D^{-1/(2d-2)}$$ as defined in \eqref{aa1044}. Note that $p(a,d,D)\le 1$ by \eqref{aa100}.

Since $\gamma \in [0,\pi/d)$, the equations \eqref{nbv1}, \eqref{nbv2} give
two possible values for $\gamma \in (0,\pi/d)$ when $p(a,d,D)<1$ (one is $\gamma=\gamma(a,d,D)$ as in \eqref{aa104} and the other is $\pi/d-\gamma$) and one possible value for $\gamma \in [0,\pi/d)$ when $p(a,d,D)=1$. (Then $\gamma=0$ if $d$ is odd, and $\gamma=\pi/(2d)$ if $d$ is even.) In the latter case, $p(a,d,D)=1$, the polynomial $f$ that attains equality in \eqref{aa101} is unique, namely $f(x)=F(x)$ with $F=R_0$ if $d$ is odd and $F=R_{\pi/(2d)}$ if $d$ is even, as defined in \eqref{ff1}.
Assume that $p(a,d,D)<1$, which corresponds to the case when the inequality in \eqref{aa100} is strict. Then, as observed above, one value
of $\gamma=\gamma(a,d,D)$ is in the open interval $(0,\pi/2d)$, the other is $\pi/d-\gamma$. So, one polynomial $f$ is for which equality in
\eqref{aa101} is attained is
\begin{equation}\label{ff1}
R_{\gamma}(x):=\prod_{k=0}^{d-1} (x-a \tan (\gamma+\pi k/d)),
\end{equation}
with $\gamma \in (0,\pi/2d)$, and the other is
$$\prod_{k=0}^{d-1} (x-a \tan (\pi/d-\gamma+\pi k/d))
= \prod_{k=0}^{d-1} (x+a \tan (\gamma+\pi k/d))=(-1)^d R_{\gamma}(-x), $$
as claimed.

It remains to show that the polynomial $R_{\gamma}$ given in \eqref{ff1} satisfies \eqref{aa102}, and to verify
\eqref{bbb23}. To prove \eqref{bbb23}, in view of \eqref{bbb}, it is sufficient to check that
\begin{equation}\label{bbb24}
(-1)^d \cot(d\pi/2+d\gamma)=\sqrt{a^{-d} 2^{2-d} d^{-d/(d-1)} D^{1/(d-1)}-1}.
\end{equation}
For $d$ odd, the left hand side of \eqref{bbb24} equals $\tan(d \gamma)$. Here, $d\gamma= \arccos p(a,d,D) \in [0,\pi/2)$ by \eqref{aa104}.  Hence $\cos(d\gamma)=p(a,d,D)$ and $\sin(d\gamma)=\sqrt{1-p(a,d,D)^2}$.
Now, taking into account \eqref{aa1044}, we obtain $$\tan(d \gamma)=\frac{\sqrt{1-p(a,d,D)^2}}{p(a,d,D)} = \sqrt{p(a,d,D)^{-2}-1}=\sqrt{a^{-d} 2^{2-d} d^{-d/(d-1)} D^{1/(d-1)}-1},$$
which yields \eqref{bbb24}.  Similarly, for $d$ even the  left hand side of \eqref{bbb24} equals $\cot(d \gamma)$, where
$d\gamma= \arcsin p(a,d,D) \in (0,\pi/2]$ by \eqref{aa104}.
 Hence $\cot(d \gamma)=\sqrt{p(a,d,D)^{-2}-1}$, which yields \eqref{bbb24} as above.

In order to show that the polynomial $R_{\gamma}$ given in \eqref{ff1} satisfies \eqref{aa102}, it suffices to prove that $R_{\gamma}(ax)=F_{a,B}(ax)$. In view of $B=\cot(d\pi/2+d\gamma)$ (see \eqref{bbb}), this is equivalent to
$$\prod_{k=0}^{d-1} (x-\tan(\gamma+k\pi/d)) = \frac{\big(1-i\cot\big(\frac{\pi d}{2}+\gamma d\big)\big)(x+i)^d+\big(1+i\cot\big(\frac{\pi d}{2}+\gamma d \big)\big)(x-i)^d}{2}.$$
We will show that this is an identity that holds for each $ x \in \C$ and all $\gamma \in \R$ for which the involved tangent and cotangent functions are defined.

Indeed, both sides are monic polynomials in $x$  of degree $d$, so it suffices to show that the right hand side vanishes at
$x=\tan(\gamma+k\pi/d)$, $k=0,1,\dots,d-1$.
Let us insert $x=\tan(\gamma+k\pi/d)$ into the right hand side and multiply it by $i^{1-d} \sin(d\pi/2+d\gamma) \cos^d (\gamma+k\pi/d)$. Since
$$i\sin(d\pi/2+d\gamma)(1 \mp i\cot(d\pi/2+d\gamma))=i\sin(d\pi/2+d\gamma) \pm \cos(d\pi/2+d\gamma),$$
and
$$i^{-d}\cos^d(\gamma+k\pi/d) \big (\tan(\gamma+k\pi/d)\pm i\big)^d= \big(-i\sin(\gamma+k\pi/d) \pm \cos(\gamma+k\pi/d)\big)^d,$$
we need to verify that
$$e^{i(d\pi/2+d\gamma)} e^{-i(k\pi+d\gamma)}-e^{-i(d\pi/2+d\gamma)} (-1)^d e^{i(k\pi+d\gamma)}=0.$$
This equality clearly holds for each $k \in \Z$, since its the left hand side equals
$$e^{i \pi(d/2-k)}-(-1)^d e^{i\pi (k -d/2)}=e^{i \pi(k-d/2)}
\big(e^{i \pi (d -2k)}-(-1)^d\big) = e^{i\pi (k-d/2)}
\big(e^{i \pi d}-(-1)^d\big)=0.$$
This completes the proof of the theorem.
\end{proof}

\begin{proof}[Proof of Theorem~\ref{Vfixed}]
Consider any monic polynomial $f$ of degree $d$ with $d$ real roots and discriminant $\Delta_f$.
It follows from \eqref{aa101} that
$$m = |f(ai)| \geq (2a)^{d/2} d^{-d/(2d-2)} \Delta_f^{1/(2d-2)}.$$
Hence $m^{2d-2} \geq (2a)^{d(d-1)} d^{-d} \Delta_f$, which implies \eqref{DiscBound}.

Assume that \eqref{Vcond} is true, and that equality holds in \eqref{DiscBound}. Then
$$ \Delta_f^{1/d(d-1)} = m^{2/d} d^{1/(d-1)} (2a)^{-1} \geq a\, 2^{1-2/d} d^{1/(d-1)},$$
which yields \eqref{aa100} for the polynomial $f$. Since equality in \eqref{DiscBound} is equivalent to equality in \eqref{aa101},
the proof of this result is now completed by applying Theorem~\ref{BoundD}.
\end{proof}

\begin{proof}[Proof of Theorem~\ref{otherpart}]
Our aim is to use Theorem~\ref{generic} that solves Problem~\ref{prob2m} for $1<m \leq 2^{d-1}$,
and then complete the proof of this result by scaling $x_k \to x_k/a$, $k=1,\dots,d$.
Observe first that  for the polynomial \eqref{genf2} we have
$$m=|f(i)|=\frac{|(d-Bi)(2i)^d|}{2d}=\frac{2^{d-1}\sqrt{d^2+B^2}}{d} \geq 2^{d-1}.$$
So the only polynomial from the family \eqref{genf2} that can be useful in the case $m \leq 2^{d-1}$
is the one with $B=0$ when $m=2^{d-1}$. We already know that it is extremal by Theorem~\ref{Vfixed}. In particular, in the case
$1<m<2^{d-1}$, by Theorem~\ref{generic},  it suffices to consider polynomials described in \eqref{genf}.

We now show that there is  only one polynomial satisfying $|f(i)|=m$ in the family of polynomials \eqref{genf} from Theorem~\ref{generic}.
To find the required value of $\lambda$ we write the condition $|f(i)|=m$ for the polynomial \eqref{genf},
where $f(x)=x^d+\sum_{k=1}^{\lfloor d/2 \rfloor} c_{d-2k} x^k$ with $c_{d-2k}$ given in \eqref{ckk}, as follows:
\begin{equation}\label{voll}
m=|f(i)|=\frac{|f(i)|}{|i^d|} = 1+\sum_{k=1}^{\lfloor d/2 \rfloor} (-1)^k c_{d-2k} =
1+\sum_{k=1}^{\lfloor d/2 \rfloor} {d\choose 2k} \frac{(2k-1)!!}{\prod_{j=1}^{k} (\lambda-2d+2j+1)}.
\end{equation}
Here, the right hand side as a function in $\lambda \in (2d-3,\infty)$ is strictly decreasing from $\infty$ to $1$. Moreover,
by \eqref{sdr}, one can see that  $\lambda=2d-2$ gives the value $2^{d-1}$ for the right hand side of \eqref{voll}. Hence for each
$m \in (1,2^{d-1}]$ there is a unique $\lambda_1=\lambda_1(d,m) \in [2d-2,\infty)$ satisfying this equation, that is, the one defined by \eqref{genfff2}.
There are no solutions in the interval $\lambda \in (2d-3,2d-2)$, since then the right hand side of \eqref{voll} is strictly greater than $2^{d-1}$,
contrary to the assumpion on $m$.

The polynomial $f$ with this $\lambda_1$ must be the only solution to Problem~\ref{prob2m} for $1<m < 2^{d-1}$. As we already observed,
for $m=2^{d-1}$ one obtains the polynomial \eqref{lopas} with $a=1$. Replacing $m$
by $m'=m a^d$, equality $m=|f(i)|$ by $m'=|f(ai)|$, and the collection of the numbers $x_1,\dots,x_d$ by the collection
$x_1'=ax_1,\dots,x_d'=ax_d$, we arrive at the unique extremal polynomial in Problem~\ref{prob2}
for $m'=ma^d$ satisfying $a^d<m' \leq 2^{d-1}a^d$. By \eqref{genfff2} and \eqref{voll},
we have $\lambda_0(a,d,m')=\lambda_1(d,m)$. Also, by \eqref{genfff} and \eqref{genf} (see also \eqref{ckk}),
the extremal polynomials $f$ (for Problem~\ref{prob2m}) and $G$ (for Problem~\ref{prob2})
are related by the formula $f(x)=G(ax)/a^d$.
Finally, \eqref{genfff222} has been established in Section~\ref{Aux22}.
\end{proof}

\begin{proof}[Proof of Corollary~\ref{smalld}]
For $d=2$, equation \eqref{genfff2} implies $\lambda_0=1+1/(m-1)$.
Inserting this value into \eqref{genfff} (with $a=1$)
we find that $G(x)=x^2+1-m$. Clearly, its discriminant is $4(m-1)$ and
its roots are $\pm \sqrt{m-1}$.

For $d=3$, equation \eqref{genfff2} implies $3/(\lambda_0-3)=m-1$. Inserting
this value into \eqref{genfff} (with $a=1$)
we find that $G(x)=x^3+(1-m)x$.  Its discriminant is $4(m-1)^3$ and
its roots are $0, \pm \sqrt{m-1}$. (Note that, by \eqref{genfff222} with $d=3$, one has $\Delta_{G}=2^2 \cdot 3^3/(\lambda_0-3)^3$ with
$\lambda_0=3+3/(m-1)$, which also gives the discriminant $4(m-1)^3$.)

In the case $d=4$, \eqref{genfff2} gives
$$m-1=\frac{6}{\lambda_0-5}+\frac{3}{(\lambda_0-5)(\lambda_0-3)}=\frac{3(2\lambda_0-5)}{(\lambda_0-5)(\lambda_0-3)},$$
which reduces to
$$\lambda_0^2-\frac{8m-2}{m-1} \lambda_0 +\frac{15m}{m-1}=0.$$ The only root satisfying $\lambda_0 \geq 2d-2=6$ is
\begin{equation}\label{lamlam}
\lambda_0 = \frac{4m-1 + \sqrt{m^2+7m+1}}{m-1}.
\end{equation}

With $\lambda_0$ satisfying \eqref{lamlam}, one can easily find that
the polynomial \eqref{genfff} is
$$
G(x)=x^4-\frac{2}{5} \big(m-4+\sqrt{m^2+7m+1}\big)x^2 +
\frac{3m+3-2\sqrt{m^2+7m+1}}{5}
$$
as in \eqref{ff}.
The discriminant of the polynomial $x^4+c_2x^2+c_0$
is equal to
\begin{equation}\label{quart}
256c_0^3-128c_2^2c_0^2+16c_2^4c_0.
\end{equation}
Inserting the values of $c_2$ and $c_0$ in terms of $m$ as in \eqref{ff}, by a computation with Maple,
we deduce that $\Delta_G$ equals to the value given in \eqref{large}.
(Alternatively, one can use the formula \eqref{genfff222} with $d=4$, $a=1$
and $\lambda=\lambda_0$ as given in \eqref{lamlam}.)
We remark that the value of $\Delta_G$ given in \eqref{large} is $0$ at $m=1$ and $2^{14}$ at $m=8$.
The latter coincides with the right hand side of \eqref{DiscBound} for $m=8$ and $d=4$.

Finally, consider the case $d=5$. The polynomial \eqref{genfff}
with $a=1$ is
$$G(x)=x^5-\frac{10}{\lambda-7} x^3 +\frac{15}{(\lambda-7)(\lambda-5)} x.$$
From \eqref{genfff2}, we deduce that
$$\lambda_0^2-\frac{12m-2}{m-1} \lambda_0 +\frac{35m}{m-1}=0,$$
and hence the only value $\lambda_0 \geq 2d-2=8$ is
$$\lambda_0=\frac{6m-1+\sqrt{m^2+23m+1}}{m-1}.$$
With this value of $\lambda_0$, we find $G$ as in \eqref{ff22}. The discriminant of the polynomial
$x^5+c_2x^3+c_0x$ is equal to the discriminant of $x^4+c_2x^2+c_0$ (as in \eqref{quart}) multiplied by $c_0^2$, that is,
 $$(256c_0^3-128c_2^2c_0^2+16c_2^4c_0)c_0^2.$$
 Inserting the values of the coefficients as in \eqref{ff22}, we find (with Maple again) that the value of the discriminant is as in \eqref{large22}.
As above, we remark that  the value of $\Delta_G$ given in \eqref{large22} is equal to $12800000=2^{12} \cdot 5^5$ at $m=16$.
The latter coincides with the right hand side of \eqref{DiscBound} for $m=16$ and $d=5$.
\end{proof}

\subsection{Proofs for Section \ref{Apps}}

\begin{proof}[Proof of Corollary~\ref{Bnonzero11}]
Note that $a_0$ defined in \eqref{aa1} coincides with the right hand side of \eqref{aa100}. By Theorem~\ref{BoundD}, one has equality in \eqref{aa101} only for the polynomial $F$ given
in \eqref{aa102} and \eqref{aa103}. Note that in the extremal case $a=a_0$ the choices for $\gamma \in [0,\pi/(2d)]$ in \eqref{aa104} are the following:
$\gamma=0$ if $d$ is odd and $\gamma=\pi/(2d)$ if $d$ is even. In both cases, \eqref{bbb} implies that
$B=0$. By Theorem~\ref{BoundD}, equality in \eqref{aa101} holds for the polynomial
$$F(x)=\frac{(x+a_0i)^d+(x-a_0i)^d}{2} \in K(d,D),$$
as stated in \eqref{aa0.5}, and we have
$$|F(a_0i)|^2=\frac{(2a_0)^{d} D^{1/(d-1)}}{d^{d/(d-1)}}=
\frac{4 d^{-d/(d-1)} D^{1/(d-1)}D^{1/(d-1)}}{d^{d/(d-1)}}=\frac{4D^{2/(d-1)}}{d^{2d/(d-1)}}.$$
This yields \eqref{aa0}.
\end{proof}

\begin{proof}[Proof of Theorem \ref{FatLem}]
Let $f \in K(d,D)$. Due to the symmetry of $E(f)$ with respect to the real line, the largest
closed disk contained in $E(f)$ must be centered on the real line. Assume this largest disk is $\{z:|z-c|\le r\}$, where $r=r(d,D)$ and $c \in \R$.  The result of Pommerenke
\cite{Pom1} implies that the point $c+ir$ is on the boundary of $E(f)$, which means that $|f(c+ir)|=1.$ For the polynomial
$g(x)=f(c+x)$ we have $g \in K(d,D)$ and $|g(ri)|=1$. This proves the first claim. To show \eqref{gneral} we observe that, by  \eqref{aa101},
$$
1 = |g(ri)| \geq (2r)^{d/2} d^{-d/(2d-2)} D^{1/(2d-2)}.$$
It follows that $2r D^{1/d(d-1)}\le d^{1/(d-1)}$, which yields \eqref{gneral}.

(i) Suppose $1 \leq D \leq 2^{1-d}\, d^d$. From $D \ge 1$ and \eqref{gneral} it follows that $r(d,D) \le d^{1/(d-1)}/2$.
Let $f$ be defined by \eqref{aa0.5}.  Then, by
$D\le 2^{1-d}\, d^d$,
\eqref{aa0} gives that
\begin{equation*}
|f(a_0i)|=2d^{-d/(d-1)}D^{1/(d-1)} \le 1.
\end{equation*}
The latter inequality means that $a_0 i \in E(f)$, and hence this point belongs to one of the closed disks centered on the real line
that are contained in $E(f)$, according to the results of Pommerenke \cite{Pom1}. Thus the radius of this disk is greater than or equal to $a_0$. Since $a_0 \geq 2^{-1+2/d} d^{-1/(d-1)}$ holds by \eqref{aa1}, the lower bound $r(d,D) \geq 2^{-1+2/d} d^{-1/(d-1)}$ follows.

(ii) Setting $$\kappa_{d,D}:=\frac{\log D}{d(d-1)},$$ we can rewrite  \eqref{gneral} in the form
\begin{equation}\label{plm}
r=r(d,D) \leq \frac{d^{1/(d-1)}} {2e^{\kappa_{d,D}}}.
\end{equation}
From $\lim_{d \to \infty} |\log D|/d^2 = 0$ it follows that $\lim_{d\to\infty} D^{1/d(d-1)} = 1$, so that $\lim_{d \to \infty} |\kappa_{d,D}|=0$. Combined with \eqref{plm}, this yields
\begin{equation}\label{plm1}
\limsup_{d \to \infty} r(d,D) \leq 1/2.
\end{equation}

By Theorem~\ref{BoundD}, we have equality in \eqref{aa101} for the polynomial $f$ of the form \eqref{aa102} under the assumption \eqref{aa100}. Hence $f(ri) \in E(f)$ if
$$(2r)^{d/2} d^{-d/(2d-2)} D^{1/(2d-2)} \leq 1.$$
This is equivalent to $2r \leq d^{1/(d-1)} e^{-\kappa_{d,D}}$, whereas \eqref{aa100}
is equivalent to $2r \leq 2^{2/d} d^{-1/(d-1)} e^{\kappa_{d,D}}$.
So $E(f)$ contains a disk with center on the real line and radius
$$\frac{1}{2}\min(d^{1/(d-1)} e^{-\kappa_{d,D}}, 2^{2/d} d^{-1/(d-1)} e^{\kappa_{d,D}}).$$ Under our assumptions on $\kappa_{d,D}$, the latter quantity tends to $1/2$ as $d \to \infty$. Hence $$\liminf_{d \to \infty} r(d,D) \geq 1/2.$$ Combined with \eqref{plm1}, this yields the result as claimed in (ii).

(iii) This time $\kappa_{d,D}$ tends to $\infty$
as $d \to \infty$, so the required result $\lim_{d \to \infty} r(d,D)=0$ follows immediately from \eqref{plm}.
\end{proof}

\begin{proof}[Proof of Theorem \ref{Equil}]
We apply Theorems \ref{Vfixed} and \ref{otherpart}, together with connecting relations
$$v = U^{\tau_d}(ai) = - \frac{\log |f(ai)|}{d} = - \frac{\log m}{d}$$
and
\[
I[\tau_d] = - \frac{1}{d(d-1)} \sum_{j\neq k} \log|x_j-x_k| = - \frac{1}{d(d-1)} \log|\Delta_f|.
\]
It is immediate to see that \eqref{EnBound} is equivalent to \eqref{DiscBound}, and \eqref{v1} is equivalent to \eqref{Vcond}. Hence
the first part of this theorem follows from Theorem \ref{Vfixed}. For the second part, we observe that \eqref{v2} is equivalent to \eqref{Vucond}
of Theorem \ref{otherpart}, which implies \eqref{EnBound2} and \eqref{Leq} together with the equality case.
\end{proof}

\begin{proof}[Proof of Corollary \ref{Lim}]
Let $\phi$ be any continuous function on $\R$ with compact support contained in $[b,c]$. Since
the minimum energy points are described by \eqref{aa103}, we deduce that
\begin{align*}
\int \phi\,d\tau_d = \frac{1}{d} \sum_{k=1}^{d} \phi(x_k) = \sum_{b\le x_k \le c} \frac{\phi(x_k)}{d}
= \sum_{b\le x_k \le c} \phi(x_k) (\arctan(x_{k+1}/a)-\arctan(x_k/a)).
\end{align*}
The latter sum can be recognized as an integral sum for the following Stieltjes integral
\begin{align*}
\int_{b}^{c} \phi(x)\,d(\arctan(x/a)) = \int_{b}^{c} \phi(x)\,\frac{a\,dx}{a^2+x^2}.
\end{align*}
Hence
\begin{align*}
\lim_{d\to\infty} \int \phi\,d\tau_d = \int_{b}^{c} \phi(x)\,\frac{a\,dx}{a^2+x^2} = \int_\R \phi(x)\,\frac{a\,dx}{a^2+x^2},
\end{align*}
and \eqref{wlim} holds by the definition of the weak* convergence.
\end{proof}

\medskip
{\bf Acknowledgement.}
The research of the first named author was funded by the European Social Fund according to the activity ‘Improvement of researchers’ quali\-fication by implementing
world-class R\&D projects’ of Measure  No. 09.3.3-LMT-K-712-01-0037. Research of the second author was partially supported by NSF via the American Institute of
Mathematics, and by the College of Arts and Sciences of Oklahoma State University.

\end{document}